\documentclass{amsart}

\usepackage{graphicx,amsmath, verbatim, amssymb, color,mathtools,hyperref,tikz}
\usepackage[margin=1.40in]{geometry}
\allowdisplaybreaks

\newtheorem{theorem}{Theorem}[section]

\newtheorem{corollary}[theorem]{Corollary}

\newtheorem{proposition}[theorem]{Proposition}

\newtheorem{question}[theorem]{Questions}

\theoremstyle{remark}
\newtheorem{remark}[theorem]{Remark}

\newtheoremstyle{TheoremNum}
        { 1em}{1 em}              
        {\itshape}                      
        {}                              
        {\bfseries}                     
        {.}                             
        { 0.5 em}                             
        {\thmname{#1}\thmnote{ \bfseries #3}}
    \theoremstyle{TheoremNum}

 \newtheoremstyle{TheoremNum}
        { 1em}{1 em}              
        {\itshape}                      
        {}                              
        {\bfseries}                     
        {.}                             
        { 0.5 em}                             
        {\thmname{#1}\thmnote{ \bfseries #3}}
    \theoremstyle{TheoremNum}

     \newtheoremstyle{TheoremNum}
        { 1em}{1 em}              
        {\itshape}                      
        {}                              
        {\bfseries}                     
        {.}                             
        { 0.5 em}                             
        {\thmname{#1}\thmnote{ \bfseries #3}}
    \theoremstyle{TheoremNum}

\usepackage[latin1]{inputenc}
\usepackage{amsmath,tikz-cd}
\usepackage{setspace}
\usepackage{amsfonts, mathrsfs}
\usepackage{amssymb}

\usepackage{pdflscape}

\numberwithin{equation}{section}

\begin{document}

\title{The degree of irrationality of most abelian surfaces is 4}
\author{Olivier Martin}

\thanks{The author acknowledges the support of the Natural Sciences and Engineering Research Council of Canada (NSERC). O. Martin est partiellement financ\'e par le Conseil de recherches en sciences naturelles et en g\'enie du Canada (CRSNG)}


\address{Department of Mathematics, University of Chicago, IL, 60637}
\email{olivier@uchicago.edu}

\maketitle

\begin{abstract}
The degree of irrationality of a smooth projective variety $X$ is the minimal degree of a dominant rational map $X\dashrightarrow \mathbb{P}^{\dim X}$. We show that if an abelian surface $A$ over $\mathbb{C}$ is such that the image of the intersection pairing $\text{Sym}^2NS(A)\to \mathbb{Z}$ does not contain $12$, then it has degree of irrationality $4$. In particular, a very general $(1,d)$-polarized abelian surface has degree of irrationality $4$ provided that $d\nmid 6$. This answers two questions of Yoshihara by providing the first examples of abelian surfaces with degree of irrationality greater than $3$, and showing that the degree of irrationality is not isogeny-invariant for abelian surfaces. 
\end{abstract}

\section{Introduction}

The \emph{degree of irrationality} $\text{irr}(X)$ of a smooth projective variety $X$ is the minimal degree of a dominant rational map $X\to \mathbb{P}^{\dim X}$. It is a birational invariant that measures how far $X$ is from being rational. While we have a relatively good understanding of the degree of irrationality of very general hypersurfaces of large degree (see \cite{BDPELU}), comparatively little is known for other classes of complex projective varietes. Since \cite{BDPELU} exploits positivity properties of $K_X$ to provide lower bounds on measures of irrationality for $X$, it is natural to consider $K$-trivial varieties as a case of particular interest. In this direction, \cite{P},\cite{AP},\cite{V},\cite{M}, and \cite{CMNP} use rational equivalence of zero-cycles on abelian varieties to obtain lower bounds on measures of irrationality for very general abelian varieties. These articles all use some form of induction on the dimension of the abelian variety, with abelian surfaces as base case. Consequently, they are powerless to provide bounds on measures of irrationality for abelian surfaces.\\

The degree of irrationality of any abelian surface is at least 3 by a result of Alzati and Pirola \cite{AP2}. Tokunaga and Yoshihara also proved that this bound is sharp in \cite{TY}. They show that if an abelian surface contains a smooth curve of genus 3, then it admits a degree 3 dominant rational map to $\mathbb{P}^2$. In particular, a very general $(1,2)$-polarized abelian surface has degree of irrationality $3$. For the same reason, the degree of irrationality of the product of two isogenous non-CM elliptic curves $E $ and $E'$ is $3$ if the minimal degree of an isogeny between $E$ and $E'$ is not $1,3,5,9,11,15,21,29,35,39,51,65,95,105,165,231$, or one other odd number if the generalized Riemann hypothesis is false \cite{RS}. Tokunaga and Yoshihara also provide an example of an elliptic curve $E$ with complex multiplication such that $E\times E$ contains a smooth genus $3$ curve as well as an example of a Jacobian of a genus $2$ curve which has degree of irrationality $3$. This prompted Yoshihara to ask in Problem 10 of \cite{Y} the following questions:
\begin{enumerate}
\item  Is there an abelian surface $A$ satisfying $\text{irr}(A)\geq 4$?
\item Do isogenous abelian surfaces have the same degree of irrationality? 
\end{enumerate}
The most recent progress on the study of the degree of irrationality of very general abelian surfaces is due to Chen who showed in \cite{C} that it is at most $4$, independently of the polarization type. This invalidated a conjecture of \cite{BDPELU}. Chen's result was improved on in \cite{CS} where it is shown that this bound applies to every abelian surface. The main result of this article is:

\begin{theorem}\label{main}
If $A$ is an abelian surface such that the image of the intersection pairing $\textup{Sym}^2NS(A)\to \mathbb{Z}$ does not contain $12\mathbb{Z}$, then the degree of irrationality of $A$ is $4$.
\end{theorem}

\begin{corollary}
A very general $(1,d)$-polarized abelian surface has degree of irrationality $4$ if $d\nmid 6$.
\end{corollary}

\begin{proof}
A very general abelian variety $A$ has Picard number $1$ so that $NS(A)$ is generated by the polarizing class. The image of $\textup{Sym}^2NS(A)\to \mathbb{Z}$ is thus $2d\mathbb{Z}$, where $d$ is the degree of the polarization.
\end{proof}

This theorem answers the first question of Yoshihara affirmatively and the second question negatively. Indeed, a very general $(1,2)$-polarized abelian surface has degree of irrationality $3$ and is isogenous to a $(1,d)$-polarized abelian surface of Picard number $1$ for some $d\nmid 6$. It follows that the degree of irrationality is not an isogeny invariant for abelian surfaces.

\begin{remark}
These results also hold if the degree of irrationality is replaced with the minimal degree of a dominant rational map to a surface $S$ with $CH_0(S)\cong \mathbb{Z}$. Indeed, we only use that $CH_0(\mathbb{P}^2)\cong\mathbb{Z}$ and that 
$$[\Delta_{\mathbb{P}^2}]\in [\text{pt}\times \mathbb{P}^2]+[\mathbb{P}^2\times \text{pt}]+\text{Sym}^2NS(\mathbb{P}^2)\subset H^4(\mathbb{P}^2\times \mathbb{P}^2,\mathbb{Z}),$$which is also valid for surfaces with a trivial Chow group of zero-cycles.
\end{remark}

\noindent These results motivate the following questions.

\begin{question}$\;$
\begin{enumerate}
\item What is the degree of irrationality of a very general abelian surface with a polarization of degree $1,3,$ or $6$?
\item If $E$ and $E'$ are non-isogenous elliptic curves, what is the degree of irrationality of $E\times E'$?
\end{enumerate}
\end{question}
The second question was suggested by Yoshihara as a possible approach to finding abelian surfaces with degree of irrationality $4$. It is generally believed that the answer is $4$. Unfortunately our approach fails to answer this question.\\

Theorem \ref{main} is obtained by a simple cohomological computation. Although Mumford's theorem for rational equivalence of zero-cycles on surfaces with $p_g\neq 0$ makes an appearance, our methods are completely different from those of \cite{P},\cite{AP},\cite{V}, \cite{M}, and \cite{CMNP}. In the second and final section we offer a proof of Theorem \ref{main}.

\vspace{2em}
\subsection*{Acknowledgements}
I am extremely grateful to Claire Voisin who suggested sweeping simplifications throughout. My original argument used an intersection count and a direct and lengthy computation. I am also indebted to Gian Pietro Pirola for carefully reading an early manuscript and pointing out a mistake. I would also like to thank Madhav Nori and Alexander Beilinson for useful discussions. This article was written in part during a research stay in Paris and I am much obliged to Claire Voisin and the Coll\`ege de France for their hospitality and the University of Chicago for providing funding.

\section{Proofs}
\begin{proof}[Proof of Theorem \ref{main}]By the Alzati-Pirola bound from \cite{AP2} and the results of \cite{C} and \cite{CS}, it suffices to show that, under the hypotheses of Theorem \ref{main}, there are no dominant rational maps $\varphi: A\dasharrow \mathbb{P}^2$ of degree $3$. Such a rational map gives rise to another rational map $F_\varphi: \mathbb{P}^2\dasharrow \text{Sym}^3(A)$, which associates to a generic point of $\mathbb{P}^2$ the fiber of $\varphi$ over it. We will denote by $Z'$ its image. Note that $Z'$ is a \emph{constant cycle subvariety}: it parametrizes rationally equivalent effective zero-cycles of degree $3$ on $A$. We can thus assume that $Z'\subset \text{Sym}^{3,0}(A)$, where $\text{Sym}^{3,0}(A)$ is the fiber of the summation map $\text{Sym}^3(A)\to A$ over $0_A\in A$. Let 
$$q: A^{3,0}:=\ker(A^3\to A)\to \text{Sym}^{3,0}(A)$$
be the quotient map and $Z:=q^{-1}(Z')$.\\

 We fix the isomorphism $\iota: A^2\to A^{3,0}$ given by $\iota(a,b)=(a,b,-a-b)$ and we abuse notation to write $Z$ for $\iota^{-1}(Z)$. Letting $U\subset A$ be an open on which the rational map $\varphi$ restrict to an \'etale morphism, we have
$$Z=\overline{\{(x,y)\in U^2: x\neq y, \varphi(x)=\varphi(y)\}}.$$\\

Another way to obtain this cycle is by resolving the indeterminacies of $\varphi$ to get a morphism $\widetilde{\varphi}: \widetilde{A}\to \mathbb{P}^2$ and a birational morphism $\pi: \widetilde{A}\to A$, such that $\varphi|_{U}=\widetilde{\varphi}\circ (\pi|_{U})^{-1}$. Letting $E_i, i\in I$ denote the irreducible curves in $\widetilde{A}$ that are contracted by $\widetilde{\varphi}$ and $D_i:=\widetilde{\varphi}_*(E_i)$, we obtain
\begin{align*}Z&=(\pi\times \pi)_*\left[(\widetilde{\varphi}\times \widetilde{\varphi})^{-1}(\Delta_{\mathbb{P}^2})-\Delta_{\widetilde{A}}-\sum_{(i,j)\in I^2}\mu_{i,j}\cdot  E_i\times E_j \right]\\&
=(\pi\times \pi)_*(\widetilde{\varphi}\times \widetilde{\varphi})^{*}(\Delta_{\mathbb{P}^2})-\Delta_{A}-\sum_{(i,j)\in I^2}\mu_{i,j}\cdot  D_i\times D_j, \end{align*}
for some coefficients $\mu_{i,j}=\mu_{j,i}\in \mathbb{Z}$. This gives us a lot of information about the cohomology class of $Z$. If we let $\text{pt}\subset \mathbb{P}^2$ denote a subvariety consisting of a single point, $h:=c_1(\mathcal{O}_{\mathbb{P}^2}(1))$, $h_i:=\text{pr}_i^*(h)$, and $h_i':=\text{pr}_1^*[\pi_*\circ \varphi^*(h)]$, we see that
\begin{align*}[Z]&=(\pi\times \pi)_*(\widetilde{\varphi}\times \widetilde{\varphi})^*(h_1^2+h_1h_2+h_2^2)-[\Delta_{A}]-\sum_{(i,j)\in I^2}\mu_{i,j}\cdot  [D_i\times D_j]\\
&= (\pi\times \pi)_*(\widetilde{\varphi}\times \widetilde{\varphi})^*([\text{pt}\times \mathbb{P}^2]+h_1h_2+[\mathbb{P}^2\times \text{pt}])-[\Delta_{A}]-\sum_{(i,j)\in I^2}\mu_{i,j}\cdot  [D_i\times D_j]\\
&= \deg\widetilde{\varphi}\cdot([A\times 0_A]+[0_A\times A])+h_1'h_2' -[\Delta_{A}]-\sum_{(i,j)\in I^2}\mu_{i,j}\cdot  [D_i\times D_j].
 \end{align*}
 Hence,
 $$[Z]\in 3([A\times 0_A]+[0_A\times A])-[\Delta_{A}]+\text{Sym}^2NS(A)\subset H^4(A\times A,\mathbb{Z}).$$\\

 But from our first description of $Z$, this variety parametrizes rationally equivalent effective cycles of degree $3$. By Mumford's theorem \cite{Mu}, it follows that $Z$ must be totally isotropic for the holomorphic $2$-form $\eta:=\iota^*(\omega_1+\omega_2+\omega_3)$, where $\omega\in H^0(A,\Omega^2_A)$ is a generator and $\omega_i:=\text{pr}_i^*\omega$. This imposes the additional constraint $[Z]\cdot \eta=0$ on the cycle class of $Z$. In particular $[Z]\cdot \eta\wedge \overline{\eta}=0$. Theorem \ref{main} then follows at once from the following:
\newpage
\begin{proposition}\label{coh}
If $\omega\in H^0(A,\Omega^2_A)$ is such that $\int_A \omega\wedge \overline{\omega}=1$, then 
$$\int_{A\times A}\left(3([A\times 0_A]+[0_A\times A])-[\Delta_{A}]\right)\cdot  \eta\wedge \overline{\eta}=-12$$
and
$$\int_{A\times A}\textup{pr}_1^*([C])\cdot \textup{pr}_2^*([C'])\cdot \eta\wedge \overline{\eta}=(C,C')\qquad \text{for any }[C],[C']\in NS(A).$$
\end{proposition}
\begin{proof}
Given $m_1,m_2,m_3\in \mathbb{Z}$, let $f_{(m_1,m_2,m_3)}: A\to A^3$ be the morphism given by $a\mapsto (m_1a,m_2a,m_3a)$. We have 
\begin{align*}\int_{A\times A} [A\times 0_A]\cdot  \eta\wedge \overline{\eta}&=\int_{A\times A} [A\times 0_A]\cdot \iota^*\left((\omega_1+\omega_2+\omega_3)\wedge(\overline{\omega}_1+\overline{\omega}_2+\overline{\omega}_3)\right)\\
&=\int_{A^3} {f_{(1,0,-1)}}_*([A])\cdot (\omega_1+\omega_2+\omega_3)\wedge(\overline{\omega}_1+\overline{\omega}_2+\overline{\omega}_3)\\
&=\int_{A}  {f_{(1,0,-1)}}^*\left((\omega_1+\omega_2+\omega_3)\wedge(\overline{\omega}_1+\overline{\omega}_2+\overline{\omega}_3)\right)\\
&=4\int_{A}\omega\wedge \overline{\omega}=4,\end{align*}
\noindent and by symmetry
\begin{align*}\int_{A\times A} [0_A\times A]\cdot  \eta\wedge \overline{\eta}=4.\end{align*}
Similarly, we see that
\begin{align*}\int_{A\times A} [\Delta_A]\cdot  \eta\wedge \overline{\eta}=&\int_{A^3}{f_{(1,1,2)}}_*([A])\cdot (\omega_1+\omega_2+\omega_3)\wedge(\overline{\omega}_1+\overline{\omega}_2+\overline{\omega}_3)\\
=&\int_{A } f_{(1,1,2)}^*\left( (\omega_1+\omega_2+\omega_3)\wedge(\overline{\omega}_1+\overline{\omega}_2+\overline{\omega}_3)\right)\\
=&36\int_A \omega\wedge \overline{\omega}=36.\end{align*}
Finally,
\begin{align*}\int_{A\times A}\text{pr}_1^*[C]\cdot \text{pr}_2^*[C']\cdot  \eta\wedge \overline{\eta}=&\int_{A^3}[\{(c,c',c+c'): c\in C, c'\in C'\}]\cdot (\omega_1+\omega_2+\omega_3)\wedge(\overline{\omega}_1+\overline{\omega}_2+\overline{\omega}_3)\\
=&\int_{A^3}[\{(c,c',c+c'): c\in C, c'\in C'\}]\cdot \omega_3\wedge\overline{\omega}_3\\
=&\int_{A^3}[\{(c,c',c+c'): c\in C, c'\in C'\}]\cdot \text{pr}_3^*([0_A])\\
=&(C,-C')=(C,C'),
\end{align*}
where $-C'$ is the image of $C'$ under the multiplication by $-1$ automorphism of $A$.
\end{proof}
\end{proof}

\begin{remark}
Consider an abelian surface $A$ of Picard number $1$ with a polarization $L$ of degree $d|6$. If there exists a dominant rational map $\varphi: A\dashrightarrow \mathbb{P}^2$ of degree $3$, the associated cohomology class $[Z]\in H^4(A^2,\mathbb{Z})$ is 
$$[Z]=3([A\times 0_A]+[0_A\times A])-[\Delta_{A}]+(6/d)\cdot\text{pr}_1^*c_1(L)\cdot \text{pr}_2^*c_1(L).$$
\end{remark}

\end{document}